\documentclass{amsart}
\usepackage[T1]{fontenc} 

\usepackage{amssymb}
\usepackage{amsthm}
\usepackage{amsmath}
\usepackage{xcolor}
\usepackage[matrix, arrow]{xy}
\xyoption{arrow}
\theoremstyle{plain}\newtheorem{Theorem}{Theorem}[section]
\theoremstyle{plain}
\theoremstyle{plain}\newtheorem{Corollary}[Theorem]{Corollary}
\theoremstyle{plain}\newtheorem{Lemma}[Theorem]{Lemma}
\theoremstyle{plain}\newtheorem{Proposition}[Theorem]{Proposition}
\theoremstyle{definition}\newtheorem{Definition}[Theorem]{Definition}
\theoremstyle{definition}\newtheorem{Example}[Theorem]{Example}
\theoremstyle{definition}
\theoremstyle{definition}
\theoremstyle{definition}
\theoremstyle{definition}\newtheorem{Remark}[Theorem]{Remark}
\theoremstyle{definition}

\def\CE{{\mathcal{E}}}  
\def\CF{{\mathcal{F}}}

\def\CV{{\mathcal{V}}}

\def\Aut{\mathrm{Aut}}               \def\tenk{\otimes_k}     
\def\Br{\mathrm{Br}}                 \def\ten{\otimes}

\def\dim{\mathrm{dim}}   \def\tenkP{\otimes_{kP}} \def\tenkQ{\otimes_{kQ}}

\def\Ext{\mathrm{Ext}}        \def\tenkR{\otimes_{kR}}

\def\GL{\mathrm{GL}}

\def\HH{H\!H}
\def\Hom{\mathrm{Hom}}           \def\tenkS{\otimes_{kS}}

\def\ker{\mathrm{ker}}

      \def\tenB{\otimes_B}
\def\Ind{\mathrm{Ind}}

\def\mod{\mathrm{mod}}

\def\Res{\mathrm{Res}}           
\def\res{\mathrm{res}}         

\makeindex
\title{Inverse images of block varieties} 
\author{Markus Linckelmann} 
\date{\today}

\begin{document}

\begin{abstract}
We extend a result due to Kawai on block varieties  for blocks with abelian defect 
groups to blocks with arbitrary defect groups.  This partially answers a question by J. Rickard.

\end{abstract}

\maketitle

\section{Introduction} 

Throughout this paper, $k$ is an algebraically closed field of prime characteristic $p$.
Given a finite group $G$, we set $H^*(G)=$ $H^*(G,k)=$ $\Ext^*_{kG}(k,k)$ and denote
by $\CV_G$ the maximal ideal spectrum of $H^*(G)$. For $H$ a subgroup of $G$, denote
by $\res^G_H : H^*(G)\to$ $H^*(H)$ the restriction map and by $(\res^G_H)^* : \CV_H\to$
$\CV_G$ the induced map on varieties. For $M$ a finitely generated $kG$-module, denote
by $I_G(M)$ the kernel of the algebra homomorphism $H^*(G)\to$ $\Ext^*_{kG}(M,M)$
induced by the functor $M\tenk-$ on the category $\mod(kG)$ of finitely generated
$kG$-modules. Denote by $\CV_G(M)$ the closed homogeneous subvariety of $\CV_G$ of
all maximal ideals of $H^*(G)$ which contain $I_G(M)$.  The map $(\res^G_H)^*$ sends
$\CV_H(\Res^G_H(M))$ to $\CV_G(M)$, and hence $\CV_H(\Res^G_H(M))$ is contained in the
inverse image of $\CV_G(M)$ under the map $(\res^G_H)^*$. 
By a result of Avrunin and Scott \cite[Theorem 3.1]{AvSc}, this inclusion is an equality; 
that is, we have 
$$\CV_H(\Res^G_H(M))=((\res^G_H)^*)^{-1}(\CV_G(M))\ .$$
Kawai proved in \cite[Proposition 5.2]{Kaw} a version of this result for block varieties
of blocks with abelian defect groups, and Rickard raised the question whether such a
result holds for blocks in general. The purpose of this paper is to extend Kawai's
result to a statement on  blocks with arbitrary defect groups which at least partially
answers Rickard's question and identifies the main issues that remain for a complete
answer.

Given a block $B$ of $kG$ with a defect group
$P$, an almost source idempotent $i\in$ $B^P$ and associated fusion system $\CF$ on $P$, 
we denote by $H^*(B)$ the block cohomology, identified with the subalgebra of all $\CF$-stable
elements in $H^*(P)$, and we denote by $\CV_B$ the maximal ideal spectrum of $H^*(B)$. 

For $Q$ a subgroup of $P$, we denote by $r_Q : H^*(B)\to H^*(Q)$ the
composition of the inclusion $H^*(B)\to H^*(P)$ and the restriction map $\res^P_Q : H^*(P)
\to H^*(Q)$.  We denote by $r^*_Q : \CV_Q \to \CV_B$ the map on varieties induced by $r_Q$.
For $M$ a finitely generated $B$-module, set $\CV_B(M) = r^*_P(\CV_P(iM))$; by Lemma
\ref{almostsourcevariety} this definition depends not on $i$ but only on the underlying
choice of a maximal $B$-Brauer pair. We have
an obvious inclusion $\CV_P(iM) \subseteq (r_P^*)^{-1}(\CV_B(M))$.  If $P$ is abelian, then
this inclusion is an equality, by Kawai \cite[Proposition 5.2]{Kaw}.  We are going to show
that for arbitrary $P$, this inclusion becomes an equality if the $kP$-module $iM$ is
$\CF$-stable, or more generally, if we  replace $iM$ by an $\CF$-stable
$kP$-module having $iM$ as a direct summand. 

We review the above  terminology in subsequent sections, and we refer to 
Proposition \ref{BLObiset} below for the notion of an $\CF$-characteristic
$P$-$P$-biset. The proofs of the
following statements are given in Section \ref{proofSection}.

\begin{Theorem} \label{thm1}
Let $G$ be a finite group, $B$ a block of $kG$, $P$ a defect group of $B$,  $i$ an almost 
source idempotent in $B^P$, and let $\CF$ be the fusion system on $P$ determined by $i$. 
Let $X$ be an $\CF$-characteristic  $P$-$P$-biset, and  let $M$ be a  finitely 
generated $B$-module.   For every  subgroup $Q$ of $P$ we have
$$\CV_Q(kX\tenkP iM) = (r_Q^*)^{-1}(\CV_B(M))\ .$$
\end{Theorem}

We do not have an example where the inclusion $\CV_P(iM)\subseteq$ $\CV_P(kX\tenkP iM)$
is proper. We list a number of cases where this inclusion is an equality.
Tensoring a $kP$-module $U$ by the bimodule $kX$ amounts to taking an $\CF$-stable
closure of $U$ (see Definition \ref{stablemoduleDef}). 
If $U$ is already $\CF$-stable, then we will see
in Lemma \ref{stablebisetvariety} below that
$\CV_P(kX\tenkP U)=$ $\CV_P(U)$. 
Thus Theorem \ref{thm1} has the following immediate consequence.

\begin{Corollary} \label{cor1}
Let $G$ be a finite group, $B$ a block of $kG$, $P$ a defect group of $B$, and $i$ an almost 
source idempotent in $B^P$. Let $\CF$ be the fusion system on $P$ determined by $i$,
and  let $M$ be a  finitely generated $B$-module.  
Suppose that the $kP$-module $iM$ is $\CF$-stable. 
For every  subgroup $Q$ of $P$ we have
$$\CV_Q(iM) = (r_Q^*)^{-1}(\CV_B(M))\ .$$
\end{Corollary}

It is not known whether there is always at least some almost source idempotent $i$ with
the property that $iM$ is fusion-stable for every finitely generated $B$-module $M$. 
We will see in Proposition \ref{invertiblebasis}  that this is the case if $iBi$ has a 
$P$-$P$-stable $k$-basis consisting of invertible elements in $iBi$.

\begin{Corollary} \label{cor2}
Let $G$ be a finite group, $B$ a block of $kG$, $P$ a defect group of $B$, and $i$ an almost 
source idempotent in $B^P$. Suppose that $iBi$ has a  $P$-$P$-stable $k$-basis contained in
$(iBi)^\times$. Then for every finitely generated $B$-module $M$ and any subgroup $Q$ 
of $P$ we have  
$$\CV_Q(iM) = (r_Q^*)^{-1}(\CV_B(M))\ .$$
\end{Corollary}

Barker and Gelvin conjectured in \cite{BarGel}, that every block with a defect group $P$
should indeed have an almost source algebra with a $P$-$P$-stable basis consisting 
of invertible elements.
If $\CF=N_\CF(P)$ and $i$ a source idempotent,  then it is easy to show that $iM$ is 
$\CF$-stable for any finitely generated $B$-module $M$. We deduce the following result.

\begin{Corollary} \label{cor3}
Let $G$ be a finite group, $B$ a block of $kG$, $P$ a defect group of $B$, and $i$ an almost 
source idempotent in $B^P$. Let $\CF$ be the fusion system on $P$ determined by $i$.
Suppose that $\CF=N_\CF(P)$. Then for any finitely generated $B$-module $M$ and any
subgroup $Q$ of $P$ we have 
$$\CV_Q(iM) = (r_Q^*)^{-1}(\CV_B(M))\ .$$
\end{Corollary}

It is well-known that if $P$ is abelian, then $\CF=N_\CF(P)$. Thus we obtain Kawai's
result mentioned above:

\begin{Corollary}[{Kawai \cite[Proposition 5.2]{Kaw}}]  \label{cor4}
Let $G$ be a finite group, $B$ a block of $kG$, $P$ a defect group of $B$, and $i$ an almost 
source idempotent in $B^P$. Suppose that $P$ is abelian.
Then for any finitely generated $B$-module $M$ and any
subgroup $Q$ of $P$ we have 
$$\CV_Q(iM) = (r_Q^*)^{-1}(\CV_B(M))\ .$$
\end{Corollary}

A block $B$ of $kG$ is {\it of principal type} if $\Br_Q(1_B)$ is a block of $kC_G(Q)$,
 for every subgroup $Q$ of $P$. If $B$ is a block of principal type,  then $1_B$ is an 
 almost source idempotent. Brauer's Third Main Theorem (see e. g. 
 \cite[Theorem 6.3.14]{LiBookII}) implies that the principal block of $kG$ is of principal
 type, and hence the principal block idempotent is an almost source idempotent.

\begin{Corollary} \label{cor5}
Let $G$ be a finite group, $B$ a block of $kG$, and $P$ a defect group of $B$. Suppose 
that $B$ is of principal type. 
Then for any finitely generated $B$-module $M$ and any
subgroup $Q$ of $P$ we have 
$$\CV_Q(M) = (r_Q^*)^{-1}(\CV_B(M))\ .$$
\end{Corollary}

Corollary \ref{cor5} applies of course also to the principal block $B_0$ of $kG$, but in 
that case  the block variety $\CV_{B_0}(M)$  coincides with the cohomology variety 
$\CV_G(M)$, and hence Corollary \ref{cor5} for the prinicipal block follows directly 
from the result \cite[Theorem 3.1]{AvSc} of Avrunin and Scott.

It is shown in \cite[Theorem 1.1]{BeLi}   that if $M$ is indecomposable, then there is a 
choice of a vertex-source pair
$(Q,U)$ of $M$ such that $\CV_B(M)= r^*_Q(\CV_Q(U))$. For such a choice of $(Q,U)$
we have $\CV_Q(U)\subseteq$ $(r_Q^*)^{-1}(\CV_B(M))$. 
This inclusion need not be an equality in general, but it becomes 
an equality if we replace $U$ by the $kQ$-module $kX\tenkQ U$.

\begin{Theorem}  \label{thm2}
With the notation of Theorem \ref{thm1}, suppose that $i$ is a source idempotent
and that the $B$-module $M$ is indecomposable.  Let
$(Q,U)$ be a vertex-source pair of $M$ such that $Q\leq P$, such that $U$ is isomorphic to a direct
summand of $iM$ as a $kQ$-module, and such that $M$ is isomorphic to a direct summand
of $Bi \tenkQ U$. Regard $kX$ as a $kQ$-$kQ$-bimodule. Then we have 
$$\CV_Q(kX\tenkQ U) = (r_Q^*)^{-1}(\CV_B(M))\ .$$
\end{Theorem}

By \cite[Proposition 6.3]{Likleinfour}, any indecomposable $B$-module $M$ has a vertex-source
pair $(Q,U)$ satisfying the hypotheses of Theorem \ref{thm2}. 
There are  examples where the inclusion 
$\CV_Q(U)\subseteq $  $\CV_Q(kX\tenkQ U)$ is proper, and so tensoring $U$ by 
$kX$ over $kQ$  in  Theorem \ref{thm2} is essential.  See Example \ref{ex2} below.

\begin{Remark} \label{rem1}
The block module variety $\CV_B(M)$ is defined  in \cite[4.1]{Linvar} by using an
injective algebra homomorphism from the block cohomology $H^*(B)$ to the Hochschild
cohomology of $\HH^*(B)$. Composed with the canonical algebra homomorphism
$\HH^*(B)\to \Ext^*_B(M,M)$ induced by the functor $-\tenB M$ this yields an algebra
homomorphism $H^*(B)\to \Ext^*_B(M,M)$, with kernel denoted $I_B(M)$. 
The variety $\CV_B(M)$ is then defined as the
closed homogeneous subvariety of $\CV_B$ consisting of the maximal ideals of $H^*(B)$
which contain $I_B(M)$. By results of Kawai \cite[Corollary 1.2]{Kaw} and the author
\cite[Theorem 2.1]{LinQuillen}, this definition of $\CV_B(M)$ is equal to $r^*_P(\CV_P(iM))$
whenever $i$ is an actual source idempotent. As mentioned earlier,  
Lemma \ref{almostsourcevariety} implies  that this identification of
$\CV_B(M)$ remains unchanged for almost source idempotents. 
\end{Remark}

The strategy to prove Theorem \ref{thm1} is as follows. We first observe that it suffices
to prove Theorem \ref{thm1} for $Q=P$. We then apply the Quillen stratification 
for block module varieties from \cite{LinQuillen} and adapt the steps in the proof
of Kawai's result \cite[Proposition 5.2]{Kaw}  to the situation at hand.

\section{Background on characteristic bisets} 

 \begin{Definition}[{cf. \cite[Definition 3.3.(1)]{LiMa}}]    \label{stablemoduleDef}
 Let $\CF$ be a saturated fusion system on a finite $p$-group $P$. 
 A $kP$-module $U$ is called {\em $\CF$-stable} if for every subgroup $Q$ of
 $P$ and every morphism $\varphi : Q\to P$ we have an isomorphism of $kQ$-modules
 ${_\varphi{U}}\cong$ $\Res^P_Q(U)$. Here ${_\varphi{U}}$ is the $kQ$-module which is
 equal to $U$ as a $k$-vector space, with $u\in$ $Q$ acting as $\varphi(u)$. 
 \end{Definition}
 
 For $Q$ a subgroup of a finite group $P$ and $\varphi : Q\to P$ an injective group 
homomorphism, we denote by $P\times_{(Q,\varphi)} P$ the transitive $P$-$P$-biset 
which is the quotient of $P\times P$ by the equivalence relation 
$(uv,w)\sim (u,\varphi(v)w)$, where  $u$, $w\in P$ and $v\in Q$. The stabiliser of the
image of $(1,1)$ in the set $P\times_{(Q,\varphi)} P$, regarded as a $P\times P$-set,
is the twisted diagonal subgroup $\Delta_\varphi(Q)=$ $\{(u,\varphi(u))\ |\ u\in Q\}$.
In particular, $P$ acts freely on the left and on the right of the set 
$P\times_{(Q,\varphi)} P$, and the cardinality of this set is  $|P|\cdot|P:Q|$. 

\begin{Proposition}[{\cite[Proposition~2.5]{BLO}}]  \label{BLObiset}
Let $\CF$ be a saturated fusion system on a finite $p$-group $P$. There is
a finite $P$-$P$-biset $X$ with the following properties:
\begin{enumerate}
\item[{\rm (i)}]  Every transitive $P$-$P$-subbiset of $X$ is of the form
$P \times_{(Q, \varphi)} P$ for some subgroup $Q$ of $P$ and some
$\varphi\in\Hom_\CF(Q,P)$.
\item[{\rm (ii)}]  $|X|/|P|$ is prime to $p$. 
\item[{\rm (iii)}]  For any subgroup $Q$ of $P$ and any $\varphi : Q\rightarrow P$
we have an isomorphism of $Q$-$P$-bisets ${_\varphi X}\cong {_Q X}$
and an isomorphism of $P$-$Q$-bisets $X_\varphi \cong X_Q$.
\end{enumerate}
\end{Proposition}

Here ${_\varphi X}$ is the $Q$-$P$-biset which as a right $P$-set is equal to $X$, with
$u\in$ $Q$ acting on the left as $\varphi(u)$ on $X$. The $P$-$Q$-biset $X_\varphi$ is
defined analogously.  The properties (i) and (iii) of $X$ in Proposition \ref{BLObiset}
do not change if we replace $X$ by a disjoint union of
finitely many copies of $X$, and therefore there exists a biset $X$
satisfying the properties (i), (iii) and (ii) replaced by the stronger requirement
$|X|/|P| \equiv 1\ (\mod\ p)$. 
Since a $P$-$P$-biset of the form  $P\times_{(Q,\varphi)} P$ has cardinality 
$|P|\cdot|P:Q|$, it follows that  
$$|X|/|P|\equiv n(X)\ (\mod\ p)\ ,$$
where $n(X)$ is the number of $P$-$P$-orbits in $X$ 
of length $|P|$.  A $P$-$P$-biset $X$ satisfying Proposition \ref{BLObiset} is called an
{\em $\CF$-characteristic biset}. (Some authors use this term for bisets satsisfying
some additional properties; see e. g. \cite[Definition 2.1]{BarGel}.)
Given two $P$-$P$-bisets $X$, $X'$, we denote by $X\times_P X'$ the quotient of the
set $X\times X'$ by the equivalence relation $(xu,x')\sim (x,ux')$, where $x\in X$, $x'\in X'$,
and $u\in P$. The left and right action of $P$ on $X\times_P X'$ is induced by the left and
right action of $P$ on $X$ and $X'$ respectively. We have an obvious $kP$-$kP$-bimodule
isomorphism $kX \tenkP kX' \cong k(X\times_P X')$.  
We record some elementary   observations for future reference.

\begin{Lemma} \label{BLObisetLemma}
Let $\CF$ be a saturated fusion system on a finite $P$-group. Let $X$, $X'$ be
$\CF$-characteristic $P$-$P$-bisets, and let $Y$ be a $P$-$P$-biset satisfying the
properties (i) and (ii) of Proposition \ref{BLObiset}. Then the $P$-$P$-bisets
$X\times_P X'$ and $X\times_P Y\times_P X'$ are $\CF$-characteristic bisets.
Moreover, the $P$-$P$-bisets $X$ and $X'$ are isomorphic to subbisets of 
$X\times_P X'$.
\end{Lemma}

\begin{proof}
Let $Q$, $R$ be subgroups of $P$ and $\varphi : Q\to P$ and $\psi : R\to P$ morphisms
in $\CF$. Using the double coset decomposition $\varphi(Q)\backslash P/R$, an easy 
verification shows that $(P\times_{(Q,\varphi)} P)\times_P (P\times_{(R,\psi)} P)$ is
a unions of $P$-$P$-orbits of the form $P\times_{(S,\tau)} P$ for some subgroup
$S$ of $P$ and some morphism $\tau :S\to P$. This implies that the bisets
$X\times_P X'$ and $X\times_P Y \times_P X'$ satisfy property (i) of Proposition
\ref{BLObiset}. One easily checks that $n(X\times_P X') = n(X)\cdot n(X')$ and the
analogous statement for $X\times_P Y\times_P X'$, which implies that 
the bisets $X\times_P X'$ and $X\times_P Y \times_P X'$ satisfy property (ii) of 
Proposition \ref{BLObiset}, and clearly these two sets inherit property (iii)
of Proposition \ref{BLObiset} from $X$ and $X'$. The last statement follows from the
fact that $X$ and $X'$ have an orbit isomorphic to $P$ as a $P$-$P$-biset.
\end{proof}

\begin{Lemma} \label{FstablemoduleLemma}
Let $\CF$ be a saturated fusion system on a finite $p$-group $P$, and let $X$ be an
$\CF$-characteristic $P$-$P$-biset. Let $U$ be a finitely generated $kP$-module. 

\begin{enumerate}
\item[{\rm (i)}]
The $P$-$P$-biset $X$ has an orbit isomorphic to $P$ as a $P$-$P$-biset.

\item[{\rm (ii)}]
The $kP$-module $kX\tenkP U$ has a direct summand isomorphic to $U$.

\item[{\rm (iii)}]
Let $Q$, $R$ be  subgroups of $P$, let $S$ be a subgroup of $Q$, and let $\varphi : S\to R$
be a morphism in $\CF$. Set $Y= Q\times_{(S,\psi)} R$. Then $Y\times_R X \cong$
$Q\times_S X$ as $Q$-$P$-bisets, and $kY\tenkR kX \cong kQ\ten_{kS} kX$ as
$kQ$-$kP$-bimodules. 

\item[{\rm (iv)}]
The $kP$-module $kX\tenkP U$ is $\CF$-stable.

\item[{\rm (v)}] 
For any subgroup $Q$ of $P$ and any morphism $\varphi : Q\to P$ in $\CF$ the
$kQ$-module ${_\varphi{U}}$ is isomorphic to a direct summand of $\Res^P_Q(kX\tenkP U)$.

\item[{\rm (vi)}]
If $U$ is $\CF$-stable, then any  indecomposable direct summand of the $kP$-module 
$kX\tenkP U$ is isomorphic to a direct summand of $kP\tenkQ U$ for some subgroup 
$Q$ of $P$.
\end{enumerate}
\end{Lemma}

\begin{proof}
Since $|X|/|P|$ is prime to $p$ by Proposition \ref{BLObiset} (ii),  it follows that $X$ has 
an orbit of length $|P|$.  By Proposition \ref{BLObiset} (i), such an orbit is isomorphic to
${_\varphi{P}}$ for some $\varphi\in$ $\Aut_\CF(P)$. It follows from Proposition 
\ref{BLObiset} (iii) that $X$ has also an orbit isomorphic to $P$. This shows (i).
It follows from (i) that $kX$ has a direct summand isomorphic to $kP$ as a
$kP$-$kP$-bimodule, which implies (ii). The statements (iii) and (iv) follow from
Proposition \ref{BLObiset} (iii). Since $U$ is isomorphic to a direct summand of
$kX\tenkP U$ as a $kP$-module, it follows that ${_\varphi{U}}$ is isomorphic to a direct 
summand  of ${_\varphi{kX}\tenkP U}\cong \Res^P_Q(kX\tenkP U)$ as a $kQ$-module, 
where the last isomorphism uses the fusion stability property from Proposition
\ref{BLObiset} (iii). This shows (v). By Proposition \ref{BLObiset} (i), every indecomposable
direct summand of $kX\tenkP U$ is isomorphic to a direct summand of
$kP\tenkQ {_\varphi{U}}$ for some subgroup $Q$ of $P$ and some morphism 
$\varphi : Q\to P$ in $\CF$. Since $U$ is assumed to be $\CF$-stable, we have
$kP\tenkQ {_\varphi{U}}\cong$ $kP\tenkQ U$. Statement (vi) follows. 
\end{proof}

\section{Background on block cohomology varieties} 

For general background on cohomology varieties see 
\cite[Section 2.25ff]{Ben84}, \cite[Chapter 5]{BenII}, \cite[Chapter 9]{CTVZ}, 
and \cite[Chapter 8]{Evens}. 
We need the following well-known facts.

 \begin{Proposition}[{\cite[Propositions 8.2.1, 8.2.4 ]{Evens}, \cite[Theorem 2.26.9]{Ben84}}]
 \label{varProp}
 For any subgroup $Q$  of  a finite group $P$, any finitely generated $kP$-module 
 $U$ and any finitely  generated $kQ$-module  $V$ we have
 $$(\res^P_Q)^*(\CV_Q(\Res^P_Q(U))) \subseteq \CV_P(U)\ ,$$
 $$(\res^P_Q)^*(\CV_Q(V)) = \CV_P(\Ind^P_Q(V))\ ,$$
 $$\CV_P(\Ind^P_Q(\Res^P_Q(U))) \subseteq \CV_P(U)\ .$$
 \end{Proposition}

We adopt the following abuse of notation: if $Q$ is a subgroup of a finite group $P$ and
$U$ a finitely generated $kP$-module, then we write $\CV_Q(U)$  instead of 
$\CV(\Res^P_Q(U))$. The third inclusion in Proposition \ref{varProp} is obviously equivalent to
the inclusion
$$\CV_P(kP\tenkQ U) \subseteq \CV_P(U)\ .$$

We briefly review block theoretic background, much of which is  from
 \cite{AlBr}, \cite{BrPuloc}, \cite{Puigpoint}, referring to \cite{LiBookI}, \cite{LiBookII} for an 
 expository account. 
We assume familiarity with relative trace maps, the Brauer homomorphism
(cf.  \cite[Theorem 5.4.1]{LiBookI}), and (local)
pointed groups on $G$-algebras. One useful technical consequence of Puig's version 
\cite[Theorem 5.12.20]{LiBookI} of Green's Indecomposability Theorem 
\cite[Theorem 5.12.3]{LiBookI}  is the following observation.

\begin{Lemma} \label{relprojLemma}
Let $G$ be a finite group, $P$ a $p$-subgroup of $G$, and $i$ a primitive idempotent
in $(kG)^P$. Let $Q$ be a subgroup of $P$ which is maximal such that $\Br_Q(i)\neq 0$.
Then there is a primitive idempotent $j\in i(kG)^Qi$ such that $\Br_Q(j)\neq 0$ and such
that 
$$ikG \cong kP\tenkQ jkG$$ 
as $kP$-$kG$-bimodules.
\end{Lemma}

Let $G$ be a finite group and $B$ a block of $kG$; that is, $B=kGb$ for some primitive
idempotent $b$ in $Z(kG)$. Thus $b$ is the unit element of $B$, called the block idempotent 
of $B$. Let $P$ be a defect group of $B$; that is, $P$ is a maximal
$p$-subgroup of $G$ such that $kP$ is isomorphic to a direct summand of $B$ as a
$kP$-$kP$-bimodule. Equivalently, $P$ is a maximal $p$-subgroup of $G$ such that
$\Br_P(b)\neq 0$. An idempotent $i\in B^P$ is a source idempotent of $B$ if 
$i$ is a primitive idempotent in the algebra $B^P$ of $P$-fixed points in $B$ with respect to the
conjugation action of $P$ on $B$, such that $\Br_P(i)\neq 0$, where $\Br_P : (kG)^P\to$
$kC_G(P)$ is the Brauer homomorphism . One of the key
properties of a source idempotent $i$ in $B^P$ is that for each
subgroup $Q$ of $P$ there is a unique block idempotent $e_Q$ of $kC_G(Q)$ such that
$\Br_Q(i)e_Q=\Br_Q(i)\neq 0$ (cf. \cite[Theorem 6.3.3]{LiBookII}).  More generally, a
(not necessarily primitive) idempotent $i$ in $B^P$ is called an almost source idempotent
if for each subgroup $Q$ of $P$ there is a unique block idempotent $e_Q$ of $kC_G(Q)$ 
such that $\Br_Q(i)e_Q=\Br_Q(i)\neq 0$. By the above, a source idempotent is an almost
source idempotent. If $i$ is an almost source idempotent in $B^P$, then
$i = i_0 + i_1$ for some source idempotent $i_0$ in $B^P$ and some idempotent  $i_1$ in
$B^P$ which is orthogonal to $i_0$. The local point of $P$ containing $i_0$ is uniquely
determined by $e_P$, hence by $i$. The extra flexibility of the notion of almost source
idempotents is particularly useful if $B$ is the principal block of $kG$, because - as
mentioned earlier -  in that  case the block idempotent $1_B$ is an almost source idempotent.

\medskip
The choice of an almost  source
idempotent $i$ in $B^P$ determines a fusion system $\CF=\CF_B(P)$ on $P$ as follows. 
For $Q$ a subgroup of $P$, denote by $e_Q$ the unique block idempotent of $kC_G(Q)e_Q$
satisfying $\Br_Q(i)e_Q=\Br_Q(i)\neq 0$. The objects
of $\CF$ are the subgroups of $P$. For two subgroups $Q$, $R$ of $P$, a group
homomorphism $\varphi : Q\to R$ is a morphism in $\CF$ if and only if there exists
an element $x\in$ $G$ such that $xQx^{-1}\leq R$, $xe_Qx^{-1}=e_{xQx^{-1}}$, and
$\varphi(u)=$ $xux^{-1}$ for all $u\in$ $Q$. See \cite[Section 8.5]{LiBookII} for more
details on fusion systems of blocks and \cite{Cravenbook} for a general introduction
to fusion systems. By the results in \cite{Pulocsource}, the fusion system $\CF$ of $B$ defined
in this way can be read off the almost source algebra $iBi$ of $B$; see \cite[Theorem 8.7.4]{LiBookII}.
A subgroup $Q$ of $P$ is fully $\CF$-centralised if $|C_P(Q)|\geq |C_P(Q')|$ for any subgroup
$Q'$ of $P$ which is isomorphic to $Q$ in $\CF$.  By \cite[Proposition 8.5.3]{LiBookII}, 
$Q$ is fully $\CF$-centralised if and only if $C_P(Q)$ is a defect group of the block
$kC_G(Q)e_Q$. 

\begin{Definition}[{\cite[Definition 4.1]{Linvar}}]  \label{blockcohomologyDef}
With the notation above, the
 block cohomology $H^*(B)$ is the graded subalgebra of $H^*(P)$ consisting of
all $\zeta\in H^*(P)$ satisfying for every morphism $\varphi : Q\to R$ in $\CF$ the 
equality $\res^P_Q(\zeta) = \res_\varphi(\res^P_R(\zeta))$. Here $\res_\varphi : H^*(R)\to$
$H^*(Q)$ is the map induced by restriction along the injective group homomorphism
$\varphi : Q\to R$.
\end{Definition}

In other words, $H^*(B)$ is the limit of the contravariant  functor on $\CF$ 
sending a subgroup  $Q$ of $P$ to $H^*(Q)$ and a morphism $\varphi : Q\to R$ in $\CF$ to
the induced map $\res_\varphi : H^*(R)\to H^*(Q)$.  If $B$ is the principal block of $kG$,
then $H^*(B)\cong$ $H^*(G)$.  As mentioned in the introduction, 
for $Q$ a subgroup of $P$, we denote by $r_Q : H^*(B)\to H^*(Q)$ the
composition of the inclusion $H^*(B)\to H^*(P)$ and the restriction map $\res^P_Q : H^*(P)
\to H^*(Q)$. 

\begin{Lemma} \label{blockcohomology1}
With this notation, the following hold for every morphism $\varphi : Q\to R$ in $\CF$.

\begin{enumerate}
\item[{\rm (i)}]  We have a commutative diagram of graded algebras
$$\xymatrix{H^*(R) \ar[rr]^{\res_\varphi} & &  H^*(Q) \\
 & H^*(B) \ar[ul]^{r_Q} \ar[ur]_{r_R} & }$$
 and $H^*(B)$ is universal with this property.  
 
 \item[{\rm (ii)}] The diagram (i) induces  a commutative diagram of varieties
 $$ \xymatrix{ \CV_Q \ar[rr]^{\res_\varphi^*} \ar[rd]_{r_Q^*} & & \CV_R \ar[ld]^{r_R^*} \\
 & \CV_B & }$$ 
 
 \item[{\rm (iii)}] 
 This diagram in (ii)  restricts for any finitely generated $kR$-module $W$ to a
 commutative  diagram of  the form
 $$ \xymatrix{ \CV_Q({_\varphi{W}}) \ar[rr]^{\res_\varphi^*} \ar[rd]_{r_Q^*} &
  & \CV_R(W)  \ar[ld]^{r_R^*} \\
 & \CV_B & }$$ 
 \end{enumerate}
 \end{Lemma}
 
 \begin{proof}
 Statement (i) is just a reformulation of the definition of $H^*(B)$ as  the limit of the
 functor $Q\mapsto H^*(Q)$ on $\CF$. Statement (ii) follows from (i) by passing
 to maximal ideal spectra, and (iii) is an immediate consequence of (ii).
 \end{proof}
 
 For $Q$ a subgroup of $P$ and a finitely generated $B$-module $M$ set
 $$\CV_Q^+ = \CV_Q \backslash \cup_R (\res^Q_R)^*(\CV_R)$$
 where in the union $R$ runs over the proper subgroups of $Q$.
 Set $\CV_Q^+(iM) = \CV_Q^+ \cap \CV_Q(iM)$. The idempotent $i$ need no longer
 be primitive in $B^Q$. If $J$ is a primitive decomposition of $i$ in $B^Q$, then
 $iM = \oplus_{j\in J} jM$ is a decomposition of $iM$ as a direct sum of $kQ$-modules.
 Thus we have
 $$\CV_Q(iM) = \cup_{j \in J}\  \CV_Q(jM)$$
 For $j\in J$ set $\CV_Q^+(jM) = \CV_Q(jM)\cap V_Q^+$.
 If $j\in J$ belongs to $\ker(\Br_Q)$, then $jM$ is relatively $R$-projective for some
 proper subgroup $R$ of $Q$, and hence $\CV_Q(jM)\subseteq (\res^Q_R)^*(\CV_R)$ in
 that case. Thus 
 $$\CV_Q^+(iM) = \cup_{j\in J^+} \ \CV_Q^+(jM)$$
 where $J^+$ is the subset of all $j\in$ $J$ satisfying $\Br_Q(j)\neq 0$ (or equivalently, 
 all $j\in J$ belonging to a local point of $Q$ on $iBi$). If $Q$ is fully $\CF$-centralised, 
 then the conjugation action by $N_G(Q,e_Q)$ on $B$ permutes the local points of $Q$ 
 on $iBi$,  and hence induces an action of the group $\Aut_\CF(Q)\cong$ 
 $N_G(Q,e_Q)/C_G(Q)$  on $\CV_Q^+(iM)$ (cf. \cite[Lemma 4.11]{Linvar}).
 
 We define further the following subvarieties of $\CV_B$. We set
 $$\CV_{B,Q}(M) = r^*_Q(\CV_Q(iM))\ ,$$
 $$\CV_{B,Q}^+(M) = r^*_Q(\CV_Q^+(iM)) = \cup_{j\in J^+}\ r^*_Q(\CV_Q(jM))\ .$$
 Denote by $\CE$ a set of representatives of the $\CF$-isomorphism classes of
 fully $\CF$-centralised elementary abelian subgroups of $P$. The block variety
 version    of Quillen's cohomology stratification states the following.
 
 \begin{Theorem}[{cf. \cite[Theorem 4.2]{Linvar}}] \label{Quillenstrat}
 With the notation above, the following hold.
 
 \begin{enumerate}
 \item[{\rm (i)}]
 The variety  $\CV_B(M)$ is a disjoint union
 $$\CV_B(M) = \cup_{E\in \CE}\ \CV_{B,E}^+(M)\  .$$

 \item[{\rm (ii)}]
 For each $E\in$ $\CE$, the group $\Aut_\CF(E)$ acts on the variety 
 $\CV_E^+(iM)$ and the  map $r^*_E$ induces an inseparable isogeny
 $\CV_E^+(iM)/\Aut_\CF(E) \to \CV_{B,E}^+(M)$.
 \end{enumerate}
 \end{Theorem}
 
 The decomposition in Theorem \ref{Quillenstrat} (i)  does not depend on the choice
 of $\CE$; this follows for instance from \cite[Lemma 4.7]{Linvar}. 
 
 \section{Almost source idempotents and fusion stable bisets} 
 
 Let $G$ be a finite group, $B$ a block of $kG$, $P$ a defect group of $B$ and $i$ an 
 almost source idempotent in $B^P$, and $\CF$ the fusion system of $B$ on $P$ determined
 by $i$.   Let if $i_0$ be a source  idempotent of $B$ which is contained in $iB^Pi$ (or 
 equivalently, which satisfies  $i_0i=i_0=ii_0$).
 
 As mentioned above, by \cite[Corollary 1.2]{Kaw} or
 \cite[Theorem 2.1]{LinQuillen}, the block variety $\CV_B(M)$ 
 of a finitely generated $B$-module $M$ is equal to $r^*_P(\CV_P(i_0M))$,  
 The next Lemma shows that we may use $i$ to calculate $\CV_B(M)$.
 Note that $i_0$ determines the same fusion system $\CF$ on $P$ because $\CF$
 depends only on the blocks $e_Q$ of $kC_G(Q)$ satisfying $\Br_Q(i)e_Q=\Br_Q(i)\neq 0$,
 for $Q$ any subgroup of $P$.
 
 \begin{Lemma}  \label{almostsourcevariety}
 We have $\CV_P(i_0M) \subseteq \CV_P(iM)$ and 
 $r^*_P(\CV_P(i_0M))=r^*_P(\CV_P(iM)) = \CV_B(M)$.
 \end{Lemma}

 \begin{proof}
 Clearly $i_0M$ is a direct summand of $iM$ as a $kP$-module, whence the first inclusion.
 Applying $r^*_P$ yields an inclusion  of varieties 
  $$r^*_P(\CV_P(i_0M))\subseteq r^*_P(\CV_P(iM))\ .$$
 The left side is the block variety $\CV_B(M)$ of $M$, as noted above. 
 The right side is the union of the varieties $r^*_P(\CV_P(i'M))$, where $i'$ runs over a primitive
 decomposition of $i$ in $B^P$. Thus, given a primitive idempotent $i'$ in $iB^P i$ we need
 to show that $r^*_P(\CV_P(i'M))$ is contained in $r^*_P(\CV_P(i_0M))$.
 It follows from  Lemma \ref{relprojLemma}  that 
 $i'M \cong kP \tenkR jM$ for some subgroup $R$ of $P$ and some primitive
 idempotent $j$ in $i'B^P i'$ satisfying $\Br_R(j)\neq 0$.  Thus we have 
 $$r^*_P(\CV_P(i'M)) = r^*_R(\CV_R(jM))\ .$$

 If $\gamma'$ is the point of $P$ on $B$ containing
 $i'$ and $\epsilon$ is the local point of $R$ on $B$ containing $j$, then $R_\epsilon$ is a defect
 pointed group of $P_{\gamma'}$. Denote by $\gamma$ the local point of $P$ on $B$
 containing $i_0$. Then $R_\epsilon$ is $G$-conjugate to a local pointed group contained
 in $P_\gamma$. That is, there is $x\in G$ such that
 $$R'_{\epsilon'} = {^x{R_\epsilon}} \leq P_\gamma\ .$$
 
 Let $j'\in$ $\epsilon'$. Since $R'_{\epsilon'}\leq P_\gamma$ we may choose $j'$ in
 $i_0B^{R'}i_0$. 
 The map $\varphi : R \to R'$ induced by conjugation with $x$ is a morphism in 
 the fusion system $\CF$, because $\Br_R(j)$  and $\Br_{R'}(j')$ are nonzero and
 belong by construction
 to the block algebras $kC_G(R)e_R$ and $kC_G(R')e_{R'}$, respectively, so  we have
 ${^x{e_R}}=$ $e_{R'}$. We clearly have an isomorphism of $kR$-modules 
 $jM \cong \res_\varphi(j'M)$. The commutative diagram in Lemma 
 \ref{blockcohomology1} (iii)  implies that
 $$r^*_R(\CV_R(jM)) = r^*_{R'}(\CV_{R'}(j'M)\ .$$

 Now $j'M$ is a direct summand of $i_0M$ as a $kP$-module, and hence we have
 $$r^*_{R'}(\CV_{R'}(j'M)\subseteq  r^*_{R'}(\CV_{R'}(i_0M))=
 r^*_P((\res^P_{R'})^*(\CV_{R'}(i_0M)))\ .$$
  By Proposition \ref{varProp} this is  contained in $r^*_P(\CV_P(i_0M))$, whence the result.
  \end{proof}
  
  \begin{Lemma} \label{stablebisetvarietyQ}
  Let $Q$ be a subgroup of $P$ and $U$ a finitely generated $kQ$-module.
  Let $X$ and $X'$ be $\CF$-characteristic $P$-$P$-bisets. The following hold.
  \begin{enumerate}
  \item[{\rm (i)}]  We have $\CV_Q(U) \subseteq \CV_Q(kX\tenkQ U)$.
  \item[{\rm (ii)}]  We have $\CV_Q(kX\tenkQ U) = \CV_Q(kX'\tenkQ U)$.
  \item[{\rm (iii)}]   We have $r_Q^*(\CV_Q(U)) = r_Q^*(\CV_Q(kX\tenkQ U))\ .$
  \end{enumerate}
  \end{Lemma}
  
  \begin{proof}
  It follows from Lemma \ref{FstablemoduleLemma} (i) that $X$ has a $Q$-$Q$-orbit
  isomorphic to $Q$, and hence that $U$ is isomorphic to a direct summand of
  $kX\tenkQ U$ as a $kQ$-module. This implies (i). Every $Q$-$P$-orbit of $X'$ is
  of the form $Q\ten_{(S,\varphi)} P$ for some subgroup $S$ of $Q$ and some morphism
  $\varphi : S\to P$ in $\CF$. Thus, by Lemma \ref{FstablemoduleLemma} (iii),
   every indecomposable direct summand of
  $kX'\tenkP kX\tenkQ U$ as a $kQ$-module is isomorphic to a direct summand of 
  $kQ \tenkS kX\tenkQ U$ for some subgroup $S$ of $Q$. By Proposition \ref{varProp}
  we have $\CV_Q(kQ \tenkS kX\tenkQ U)\subseteq \CV_Q(kX \tenkQ U)$.
  This shows that $\CV_Q(kX'\tenkP kX\tenkQ U)\subseteq$ $\CV_Q(kX\tenkQ U)$.
  By Lemma \ref{BLObisetLemma},  $X'$ is isomorphic to a $P$-$P$-subbiset of 
  $X'\times_P X$. Thus $kX'\tenkQ U$ is isomorphic to a direct summand of
  $kX'\tenkP kX\tenkQ U$ as a $kQ$-module, and we therefore have 
  $\CV_Q(kX'\tenkQ U) \subseteq$ $\CV_Q(kX'\tenkP kX\tenkQ U)$.  Together we
  get that $\CV_Q(kX'\tenkQ U)\subseteq$ $\CV_Q(kX\tenkQ U)$. Exchanging the 
  roles of $X$ and $X'$ shows that this inclusion is an equality, whence (ii).
  By Proposition \ref{BLObiset} (i), as a  $kQ$-module,  $kX\tenkQ U$ is isomorphic to 
 a direct sum of $kQ$-modules of the form $kQ \tenkR {_\psi{U}}$, with $R$ a subgroup
 of $Q$ and $\psi : R\to Q$ a morphism in $\CF$. By  Proposition \ref{varProp} we have
 $$\CV_Q(kQ \tenkR {_\psi{U}})=(\res^Q_R)^*({_\psi{U}})\ .$$
 Since $r^*_R = r^*_Q\circ (\res^Q_R)^*$, it follows that 
 $$r^*_Q(\CV_Q(kQ\tenkR {_\psi{U}})) = r^*_R(\CV_R({_\psi{U}})) = 
 r^*_{\psi(R)}(\CV_{\psi(R)}(U))$$
 where the last equality uses Lemma \ref{blockcohomology1} (iii). 
 Using Proposition \ref{varProp} again we get that
 $$r^*_{\psi(R)}\CV_{\psi(R)}(U)) = r^*_R((\res^Q_R)^*(\CV_{R}(U)))
 \subseteq r^*_Q(\CV_Q(U))\ .$$
 This proves (iii).  
  \end{proof}

 \begin{Lemma} \label{stablebisetvariety}
 Let $X$ be an $\CF$-characteristic $P$-$P$-biset, and let
  $U$ be a finitely generated $kP$-module. 
 If $U$ is $\CF$-stable, then $\CV_P(U) = \CV_P(kX\tenkP U)$.
 \end{Lemma}
 
 \begin{proof}
 By Lemma \ref{stablebisetvarietyQ} we have $\CV_P(U)\subseteq$ $\CV_P(kX\tenkP U)$.
 Assume that $U$ is $\CF$-stable. Let $U'$ be an indecomposable direct
 summand of $kX\tenkP U$. By Lemma \ref{FstablemoduleLemma} (vi), $U'$ is isomorphic
 to a direct summand of $kP\tenkQ U$ for some subgroup $Q$ of $P$. Thus, by
 Proposition \ref{varProp}, we have $\CV_P(U')\subseteq\CV_P(kP\tenkQ U)
 \subseteq \CV_P(U)$.   This implies $\CV_P(kX\tenkQ U)\subseteq$ $\CV_P(U)$.
 The result follows.
 \end{proof}
 
 As a $kP$-$kP$-bimodule, $iBi$ is a direct summand of $kG$.  Thus $iBi$ has a
 $P$-$P$-stable $k$-basis $Y$. 
 
 \begin{Lemma}\label{iBibiset}
 Let $Y$ be a $P$-$P$-stable basis of $iBi$. Then $Y$ has a $P$-$P$-orbit isomorphic to
 $P$, and $Y$ satisfies the property (i) from Proposition \ref{BLObiset}. If in addition 
 $i$ is a source idempotent,  then $Y$ satisfies the properties (i) and (ii) from Proposition 
 \ref{BLObiset}.
 \end{Lemma}
 
\begin{proof} This follows, for instance, from
 \cite[Propositions 8.7.10]{LiBookII} together with the fact, due to Puig, that if
 $i$ is a source idempotent, then 
 $\frac{\dim_k(iBi)}{|P|}$ is prime to $p$ (see e. g. \cite[Theorem 6.15.1]{LiBookII}). 
 \end{proof}
 
 It is not known whether $i$ can always be chosen in such a way that $Y$ is an $\CF$-characteristic
 biset. See Proposition \ref{invertiblebasis} below for a sufficient criterion for $Y$ to 
 satisfy property (iii) of Proposition \ref{BLObiset}.

 \begin{Lemma} \label{iBiX}
 Let $Q$ be a subgroup of $P$.
 As a $kQ$-$kP$-bimodule, $iBi\tenkP kX$ is isomorphic to a direct sum of bimodules
 of the form $kQ \tenkR kX$, with $R$ running over the subgroups of $Q$. Moreover,
 $iBi\tenkP kX$ has a direct summand isomorphic to $kX$ as a $kQ$-$kP$-bimodule.
 \end{Lemma}
 
 \begin{proof}
 By Lemma \ref{iBibiset} or  by \cite[Theorem  8.7.1]{LiBookII}, as a 
 $kQ$-$kP$-bimodule, $iBi$ is isomorphic
 to a direct sum of bimodules of the form $kQ \tenkR {_\psi{kP}}$, for some subgroup
 $R$ of $Q$ and some morphism $\psi : R\to P$ in $\CF$. Thus $iBi\tenkP kX$ is 
 isomorphic to a direct sum of $kQ$-$kP$-bimodules of the form $kQ \tenkR {_\psi{kX}}\cong$
 $kQ\tenkR kX$, where we use the $\CF$-stability of $X$. Since $\Br_P(i)\neq 0$, it
 follows that $iBi$ has a direct summand isomorphic to $kP$ as a $kP$-$kP$-bimodule,
 hence also as a $kQ$-$kP$-bimodule, and therefore $iBi\tenkP kX$
 has a direct summand isomorphic to $kX$ as a $kQ$-$kP$-bimodule. The result follows.
 \end{proof}
 
 \begin{Lemma} \label{ViBiX}
 Let $Q$ be a subgroup of $P$ and $W$ a finitely generated $kQ$-module. We have
 $$\CV_Q(iBi\tenkQ W) \subseteq \CV_Q(kX\tenkQ W)\ .$$
 \end{Lemma}
 
 \begin{proof}
 Note that $kX$ has a direct summand isomorphic to $kP$ as a $kP$-$kP$-bimodule.
 Thus $iBi$ is isomorphic to  a direct summand of $iBi\tenkP kX$ as a $kP$-$kP$-bimodule, 
 hence also as a $kQ$-$kQ$-bimodule, and therefore
 $$\CV_Q(iBi\tenkQ W)\subseteq 
 \CV_Q(iBi\tenkP kX \tenkQ W)\ .$$  
 By Lemma \ref{iBiX}, as a $kQ$-module,
 $iBi\tenkP kX\tenkQ W$ is isomorphic to a direct sum of modules of
 the form $kQ\tenkR kX\tenkQ W$ with at least one summand where $R=Q$.
 Thus the variety $\CV_Q(iBi\tenkP kX \tenkQ W)$
 is contained in the union of varieties of the form $\CV_Q(kQ\tenkR kX\tenkQ  W)$.
 By Proposition \ref{varProp}, these are all contained in $\CV_Q(kX\tenkQ W)$,
 proving the result. 
 \end{proof}
 
 \begin{Proposition} \label{invertiblebasis}
 Let $G$ be a finite group, $B$ a block of $kG$, $P$ a defect group of $B$ and $i$ an
 almost source idempotent in $B^P$. Suppose that $iBi$ has a $P$-$P$-stable
 $k$-basis $X$ which is contained in $(iBi)^\times$. The following hold.
 
 \begin{enumerate}
 \item[{ \rm (i)}] If $i$ is a source idempotent, then $X$ is an $\CF$-characteristic  $P$-$P$-biset.
 \item[{\rm (ii)}] For every subgroup $Q$ of $P$ and any morphism $\varphi : Q\to P$ in
 $\CF$ we have an isomorphism of $kQ$-$B$-bimodules ${_\varphi{iB}}\cong iB$.
 \item[{ \rm (iii)}] For every finitely generated $B$-module $M$  the $kP$-module
 $iM$ is $\CF$-stable. 
 \end{enumerate}
 \end{Proposition}
 
 \begin{proof}
 Statement (i) is proved for instance in  \cite[Proposition 8.7.11]{LiBookII}.
 Let $Q$ be a subgroup of $P$ and $\varphi : Q\to P$ a morphism in $\CF$.
 By Alperin's Fusion Theorem \cite[Theorem 8.2.8]{LiBookII}, in order to prove
 (ii) we may assume that $Q$ is $\CF$-centric and that $\varphi$ is an
 automorphism of $Q$ composed with the inclusion map $Q\leq P$. By
 \cite[Proposition 8.7.10]{LiBookII} there exists an element $x\in X$ such that
 $ux=x\varphi(u)$ for all $u\in $ $Q$. One checks that left multiplication by
 $x$ on $iB$ is a homomorphism of $kQ$-$B$-bimodules ${_\varphi{iB}}\to  iB$.
 Since $x$ is invertible in $iBi$, this map is an isomorphism, proving (ii).
 We have $iM\cong iB\tenB M$, so (ii) implies (iii).
 \end{proof}
 
 It is not known whether every  block $B$ with defect group $P$
 has at least some almost source idempotent
 $i\in B^P$ such that the almost source algebra $iBi$ has a $P$-$P$-stable basis
 consisting of invertible elements. See \cite{BarGel} for equivalent reformulations
 of this problem,  as well as a number of cases in which this is true. 
 The following technical observation is a special case of Puig's
 characterisation of fusion in source algebras in \cite{Pulocsource}. 
 
 \begin{Lemma} \label{Pfusiondetect}
 Let $G$ be a finite group, $B$ a block of $kG$, $P$ a defect group of $B$ and $i$ a
 source idempotent in $B^P$.  Denote by  $\CF$ the fusion system on $P$ determined
 by $i$.  Let $\varphi\in\Aut(P)$. Then $\varphi\in\Aut_\CF(P)$ if and only if 
 ${_\varphi{iB}}\cong$ $iB$ as $kP$-$B$-bimodules.
 \end{Lemma}
 
 \begin{proof}
 This is the special case of \cite[Theorem 8.7.4.(ii)]{LiBookII} applied to the case 
 where $P=Q=R$ and $i$ is an actual source idempotent.
 \end{proof}
 
 \begin{Proposition} \label{PnormalF}
  Let $G$ be a finite group, $B$ a block of $kG$, $P$ a defect group of $B$ and $i$ a
 source idempotent in $B^P$.  Denote by  $\CF$ the fusion system on $P$ determined
 by $i$ and suppose that $\CF=N_\CF(P)$. 
For every finitely generated $B$-module $M$ the $kP$-module 
 $iM$ is $\CF$-stable. 
 \end{Proposition}
 
 \begin{proof}
 Since $\CF=N_\CF(P)$, it suffices to check the fusion stability condition on $iM$ for
 automorphisms of $P$ in $\CF$. This follows from the obvious $kP$-isomorphism 
 $iB\tenB M\cong$ $iM$ and Lemma \ref{Pfusiondetect}.
 \end{proof}
 
 \begin{Lemma} \label{Mstable}
  Let $G$ be a finite group, $B$ a block of $kG$, $P$ a defect group of $B$ and $i$ an
 almost source idempotent in $B^P$.  Denote by  $\CF$ the fusion system on $P$ determined
 by $i$. For every finitely generated $B$-module $M$ the $kP$-module 
 $\Res^G_P(M)$ is $\CF$-stable. 
 \end{Lemma}
 
 \begin{proof}
 Let $Q$ be a subgroup of $P$ and $\varphi : Q\to P$ a morphism in $\CF$. Then 
 there exists an element $x\in$ $G$ such that $\varphi(u) = xux^{-1}$ for all $u\in$ $Q$.
 Then the map sending $m\in$ $M$ to $xm$ is an isomorphism of $kQ$-modules
 $\Res^G_Q(M)\cong$ ${_\varphi{M}}$.
 \end{proof}

\section{Proofs} 
\label{proofSection}

\begin{proof}[Proof of Theorem \ref{thm1}]
Set $U = kX\tenkP iM$. Note that the $kP$-module $U$ is $\CF$-stable.
By Lemma \ref{stablebisetvariety} we have 
$$\CV_B(M) = r_P^*(\CV_P(iM)) = r_P^*(\CV_P(U))$$
and hence we have
$$\CV_P(U) \subseteq (r^*_P)^{-1}(\CV_B(M))\ .$$
We observe first that it suffices to show Theorem \ref{thm1} for $Q=P$. Indeed,
suppose that 
$$\CV_P(U) = (r_P^*)^{-1}(\CV_B(M)) \ .$$
Let $Q$ be a subgroup of $P$. By \cite[Theorem 3.1]{AvSc} we have 
$$\CV_Q(U)= ((\res^P_Q)^*)^{-1}(\CV_P(U)\ .$$
Since $r_Q = \res^P_Q \circ r_P$, it follows from these two equalities that
$$(r_Q^*)^{-1}(\CV_B(M))=((\res^P)Q)^*)^{-1}( (r_P^*)^{-1}(\CV_B(M))) =
((\res^P_Q)^*)^{-1}(\CV_P(U))=\CV_Q(U)\ .$$
This shows that it suffices to prove Theorem \ref{thm1} for $Q=P$.
We need to show that the inclusion $\CV_P(U) \subseteq$$ (r^*_P)^{-1}(\CV_B(M))$
is an equality.  Let $z\in$ $(r_P^*)^{-1}(\CV_B(M))$. We need to show that $z\in$
$\CV_P(U)$.  By choice of $z$, we have $z\in \CV_P$ and 
$r^*_P(z) \in \CV_B(M)$. Quillen's stratification applied to the $kP$-module $U$ yields
$$\CV_P(U) = \cup_E\ (\res^P_E)^*(\CV_E^+(U)\ ,$$
where $E$ runs over a set of representatives of the conjugacy classes of elementary
abelian subgroups of $P$. This is a disjoint union. 

Quillen's stratification applied to $\CV_P$ implies that  $z\in$ $\CV_{P,E}^+=$
$(\res^P_E)^*(\CV_E^+)$ for some elementary abelian subgroup $E$ of $P$; that is, we have 
$$z = (\res^P_E)^*(s)$$
for some $s\in \CV_E^+$. Note that $E$ is unique up to conjugation in $P$ and $s$ is unique
up to the action of $N_P(Q)$. 

We need to show that $E$ and $s$ can be chosen in such a way that $s\in$ $\CV_E^+(U)$.
The block variety version of Quillen's stratification, reviewed in Theorem \ref{Quillenstrat}
and preceding paragraphs,   implies that 
 $$r^*_P(z)= r^*_F(t)$$
 for some fully $\CF$-centralised elementary abelian subgroup $F$ of $P$ and some
 $t\in \CV_F^+(iM)$. Applying $r^P_*$ to the first equation yields
 $$r^*_P(z) = r^*_E(s)\ .$$
 This implies that  $r^*_E(s)=$ $r^*_F(t)$ in the block variety $\CV_B$. The analogue of
 Quillen's stratification for the block variety $\CV_B$ implies that  there is an isomorphism 
 $\varphi : E \cong F$ in $\CF$ such that $w=\res_\varphi^*(s)$ and $t$ 
 are in the same $\Aut_\CF(F)$-orbit in $\CV_F^+$.  That is, after composing $\varphi$
 with a suitable automorphism of $F$, we may assume that $t=\res_\varphi^*(s)$.
 Now $t$ belongs to $\CV_F^+(iM)\subseteq\CV_F^+(U)$. The $\CF$-stability of $U$
 implies that $s\in$ $\CV_E^+(U)$. This completes the proof of Theorem \ref{thm1}.
\end{proof}

Just as for Theorem \ref{thm1} it follows from \cite[Theorem 3.1]{AvSc} that it suffices
to prove any of the five Corollaries to Theorem \ref{thm1} for $Q=P$.
Note further that thanks to Lemma \ref{almostsourcevariety} we may assume that in 
all of these Corollaries the almost source idempotent is a source idempotent

\begin{proof}[Proof of Corollary \ref{cor1}]
This follows from Theorem \ref{thm1} combined with Lemma \ref{stablebisetvariety}.
\end{proof}

\begin{proof}[Proof of Corollary \ref{cor2}]
This follows from Corollary \ref{cor1} and Proposition \ref{invertiblebasis}.
\end{proof}

\begin{proof}[Proof of Corollary \ref{cor3}]
This follows from Corollary \ref{cor1} and Proposition \ref{PnormalF} (here we
make use of the fact that $i$ can be assumed to be a source idempotent, by
Lemma \ref{almostsourcevariety}).
\end{proof}

\begin{proof}[Proof of Corollary \ref{cor4}]
Since $P$ is abelian, it is well-known that $\CF=N_\CF(P)$ (see e. g. 
\cite[Proposition 8.3.8]{LiBookII}).
Thus Corollary \ref{cor4} follows from Corollary \ref{cor3}.
\end{proof}

\begin{Remark}
It is shown in \cite[Proposition 1.7]{BarGel}  that in the situation of Corollaries
\ref{cor3}, \ref{cor4} the source algebras have $P$-$P$-stable bases consisting
of invertible elements. Thus these two corollaries follow from this combined with
Corollary \ref{cor2}.
\end{Remark}

\begin{proof}[Proof of Corollary \ref{cor5}]
By Lemma \ref{Mstable}, the restriction to $P$ of any finitely generated 
$B$-module is $\CF$-stable. Since $B$ is assumed to be of principal type,
it follows that $1_B$ is an almost source idempotent of $B$. Thus
Corollary \ref{cor5} follows from Corollary \ref{cor1}.
\end{proof}

Corollary \ref{cor5} can also be proved by combining \cite[Corollary 2.5]{BarGel} with
Corollary \ref{cor2}.

\begin{proof}[Proof of Theorem \ref{thm2}]
By \cite[Theorem 1.1]{BeLi}, we have $\CV_B(M)=$ $r^*_Q(\CV_Q(U))$, and hence
we have $\CV_Q(U)\subseteq$ $(r_Q^*)^{-1}(\CV_B(M)$. By Lemma \ref{stablebisetvarietyQ}
we have $r^*_Q(\CV_Q(U))=$ $r^*_Q(\CV(kX\tenkQ U))$, and therefore
$$\CV_Q(kX\tenkQ U) \subseteq (r^*_Q)^{-1}(\CV_B(M))\ .$$
We need to show that this inclusion is an equality. By Theorem \ref{thm1} we have
$$(r^*_Q)^{-1}(\CV_B(M)) = \CV_Q(kX \tenkP iM)\ .$$
By the choice of the vertex-source pair $(Q,U)$ of $M$, the $iBi$-module $iM$ is 
isomorphic to a direct summand of $iBi\tenkQ U$. Thus we have 
$$\CV_Q(kX\tenkP iM) \subseteq \CV_Q(kX\tenkP iBi \tenkQ U)\ .$$
Now $iBi$ is isomorphic to a direct summand of $iBi\tenkP X$ as a
$kP$-$kQ$-bimodule, and hence we get an inclusion
$$\CV_Q(kX\tenkP iBi \tenkQ U)\subseteq  \CV_Q(kX\tenkP iBi \tenkP kX \tenkQ U)\ .$$
Let $Y$ be a $P$-$P$-stable $k$-basis of $iBi$, so that $iBi\cong kY$ as $kP$-$kP$-bimodule.
By Lemma \ref{iBibiset}, $Y$ satisfies the properties (i) and (ii) from Proposition 
\ref{BLObiset}. It follows from 
Lemma \ref{BLObisetLemma}, that the set $X\times_P Y \times_P  X$ is an $\CF$-characteristic
$P$-$P$-biset. Thus, by Lemma \ref{stablebisetvarietyQ} we have an equality
$$\CV_Q(kX\tenkP iBi \tenkP kX \tenkQ U) = \CV_Q(kX \tenkQ U)\ .$$
Together this shows the inclusion
$$(r^*_Q)^{-1}(\CV_B(M) \subseteq  \CV_Q(kX \tenkQ U)\ .$$
This completes the proof of Theorem \ref{thm2}. 
\end{proof}

\section{Examples} 

With the notation of Theorem \ref{thm1}, we do not know of an example where the
inclusion $\CV_P(iM)\subseteq$ $\CV_P(kX\tenkP iM)$ is strict. The following example
constructs  a finitely generated $kP$-module $U$ such that the
inclusion $\CV_P(U) \subseteq$ $\CV_P(kX\tenkP U)$ is strict. 

\begin{Example} \label{ex1}
Suppose that $p$ is odd.  Let $Q$, $R$ be cyclic groups of order $p$, and  let $u$, $v$ be a 
generator of $Q$, $R$, respectively. Set $P=Q\times R$. Let $\tau$ be the automorphism of
order $2$ of $P$ which exchanges $u$ and $v$ (identified to their images in $P$).
Set $V=\Ind^P_Q(k)$ and $W=\Ind^P_R(k)$. Since $\tau$ exchanges $Q$ and $R$, it
follows that $V$ and $W$ are exchanged by $\tau$; that is, $W\cong {_\tau{V}}$ and
$V\cong {_\tau{W}}$.  Set $L=P\rtimes\langle\tau\rangle$ and denote by $\CF$ the
fusion system of $L$ on $P$.  We have
$$\Res^L_P\Ind^L_P(V)\cong \Res^L_P\Ind^L_P(W) \cong V\oplus W\ .$$
By Proposition \ref{varProp} we have
$$\CV_P(V) = (\res^P_Q)^*(\CV_Q)\ ,$$
$$\CV_P(W)=(\res^P_R)^*(\CV_R)\ .$$
Since $Q$, $R$ are different cyclic subgroups of $P$, the varieties $\CV_P(V)$ and
$\CV_P(W)$ are different lines in $\CV_P$.
Note that $kL$ has a unique block $B=kL$ and
 that $H^*(L)=$ $H^*(B)$ is the subalgebra of $\tau$-stable elements in $H^*(P)$, or
 equivalently, the subalgebra of $\CF$-stable elements in $H^*(P)$.  The $P$-$P$-biset
 $X=L$ is an $\CF$-characteristic biset. Since $L=P\cup P\tau$, it follows that
 $$kX \tenkP V = V\oplus W\ \cong \Res^L_P\Ind^L_P(V)\ $$
 from which we get a strict inclusion
 $$\CV_P(V) \subseteq \CV_P(kX\tenkP V) = \CV_P(V)\cup \CV_P(W)\ .$$ 
 Denote by
 $r_P : H^*(L) \to H^*(P)$ the inclusion map, and by $r_P^* : \CV_P \to \CV_L$ the induced
 map on varieties.  By Proposition \ref{varProp} we have
 $$r^*_P(\CV_P(V)) = \CV_L(\Ind^L_P(V)) = r^*_P(\CV_P(W))\ .$$
 By  \cite[Theorem (3.1)]{AvSc}, applied to $\Ind^L_P(V)$, we have
 $$(r^*_P)^{-1}(\CV_L(\Ind^L_P(V))) = \CV_P(V\oplus W) = \CV_P(V) \cup \CV_P(W)\ .$$
Since
the action on $\CV_P$ induced by $\tau$ exchanges $\CV_P(V)$ and $\CV_P(W)$,
it follows that $r_P^*(\CV_P(V))=r^*_P(\CV_P(W))$.  Thus $\CV_P(V)$  and $\CV_P(W)$
are both contained in $(r^*_P)^{-1}(r^*_P(\CV_P(V)))$. This shows that we have a strict inclusion
$\CV_P(V) \subseteq  (r^*_P)^{-1}(r^*_P(\CV_P(V)))$. 
\end{Example}

\begin{Remark}
The Example \ref{ex1} contradicts the inclusion $\supseteq$ in the statement of
\cite[Theorem 2.2]{To14}.
While the inclusion $\subseteq$ in \cite[Theorem 2.2]{To14} holds in the generality as
stated there, for the reverse inclusion one needs some extra hypotheses. With the notation
of \cite[Theorem 2.2]{To14}, the following hypotheses, communicated to the author by
C.-C. Todea, are sufficient for the reverse inclusion:  $\CF_1$and  $\CF_2$ are
saturated fusion systems of finite groups $G_1\leq $ $G_2$ on $P_1\leq$ $P_2$ and
$U$ is a finitely generated $kG_2$-module. 
\end{Remark}

\begin{Example} \label{ex2}
We adapt the previous example to show that tensoring by $kX$ over $Q$ in
Theorem \ref{thm2} is necessary if $Q$ is a proper subgroup of $P$, even possibly when 
$B$ is a nilpotent block.
Let $p=2$ and $Q$ be a Klein four group. Write
$Q=\langle s\rangle \times \langle t\rangle$ with involutions $s$, $t$. The group
$\GL_2(k)$ acts on $kQ$ in the obvious way (by sending $s$, $t$ to shifted cyclic
subgroups). Let $W=kQ/\langle t\rangle$; this is a $2$-dimensional $kQ$-module
with vertex $\langle t\rangle$, hence periodic of period $1$. 
Since there are only finitely many isomorphism classes
of $kQ$-modules with cyclic vertex, it follows that ${_\tau{W}}$ has vertex $Q$ for almost
all $\tau \in$ $\GL_2(k)$.  Set $P=Q\rtimes \langle u\rangle$
for some involution $u$ satisfying $usu=t$ (so that $P$ is a dihedral group). 
Choose $\tau\in$ $\GL_2(k)$ such that $U={_\tau{W}}$ has vertex $Q$ and such that
$c_u\circ \tau\neq \tau$, where $c_u$ is conjugation by $u$ regarded as an automorphism
of $kQ$.  Set $M=\Ind^P_Q(U)$ and $U'={_{c_u}{U}}$. Then $\Res^P_Q(M)\cong U\oplus U'$.
Both $(Q,U)$ and $(Q,U')$ are vertex-source pairs of $M$. 
Since $U$, $U'$ are periodic, the choice of $\tau$ implies that the varieties  $\CV_Q(U)$ 
and $\CV_Q(U')$ are different lines in $\CV_Q$. 
The fusion system $\CF$ is in this situation
the trivial fusion system $\CF_P(P)$, and the set $X=$ $P$, as a $P$-$P$-biset, is a 
characteristic biset of $\CF$. Thus, as a $kQ$-module, we have $kX\tenkQ U\cong$ 
$\Res^P_Q(\Ind^P_Q(U))\cong$ $U\oplus U'$,
and since the varieties $\CV_Q(U)$ and $\CV_Q(U')$ are different, it follows that
$\CV_Q(U)$ is properly contained in $\CV_Q(kX\tenkQ U)$.
\end{Example}


\end{document}